\documentclass[11pt, a4paper]{article}
\usepackage{amsmath,amssymb,latexsym,color,enumerate,amsthm}
\usepackage{graphicx}
\usepackage{indentfirst}
\usepackage{authblk}
\usepackage[mathscr]{eucal}
\usepackage[colorlinks,
linkcolor=blue,
anchorcolor=green,
citecolor=magenta
]{hyperref}
\usepackage{setspace}
\newtheorem{thm}{Theorem}[section]

\newtheorem{lemma}[thm]{Lemma}
\newtheorem{cor}[thm]{Corollary}

\newtheorem{claim}[thm]{Claim}

\newtheorem{example}{Example}[section]
\newtheorem{defin}{Definition}[section]

\def\q{\hfill\rule{1ex}{1ex}}

\usepackage{CJK}
\begin{document}
	\title{\bf Edge pancyclic Cayley graphs on symmetric group}
	
	\author[1]{Mengyu Cao}
	\author[2]{Mei Lu}
	\author[2]{Zequn Lv}
	\author[2]{Xiamiao Zhao\thanks{Corresponding author. E-mail: \texttt{zxm23@mails.tsinghua.edu.cn}}}
	
	\affil[1]{\small Institute for Mathematical Sciences, Renmin University of China, Beijing 100086, China}
	\affil[2]{\small Department of Mathematical Sciences, Tsinghua University, Beijing 100084, China}

	\date{}

	\maketitle\baselineskip 16.3pt
	
	\begin{abstract}

	We study the derangement graph $\Gamma_n$ whose vertex set consists of all permutations of $\{1,\ldots,n\}$, where two vertices are adjacent if and only if their corresponding permutations differ at every position. It is well-known that $\Gamma_n$ is a Cayley graph, Hamiltonian and Hamilton-connected. In this paper, we prove that for $n \geq 4$, the derangement graph $\Gamma_n$ is edge pancyclic. Moreover, we extend this result to two broader classes of Cayley graphs defined on symmetric group.
		
		\end {abstract}
		
		\vskip.3cm
		
		\section{Introduction}
		
		Let \(\Gamma = (V, E)\) be a graph. For any subset \(S \subseteq V\), \(\Gamma[S]\) denotes the subgraph of \(\Gamma\) induced by \(S\). For each \(v \in V(\Gamma)\), let \(N(v) = \{w \mid vw \in E(\Gamma)\}\) be the neighborhood of \(v\), and let \(d(v) = |N(v)|\) denote the degree of \(v\). Let \(\delta = \delta(\Gamma)\) represent the minimum degree of \(\Gamma\). A matching of size \(s\) in \(\Gamma\) is a set of \(s\) pairwise disjoint edges, and if it covers all vertices of \(\Gamma\), it is called a \emph{perfect matching}. 
		
		For a graph \(\Gamma\) with order \(n \geq 3\), we say that \(\Gamma\) is \emph{Hamiltonian} if it contains a cycle that spans all the vertices in \(V(\Gamma)\). We say that \(\Gamma\) is \emph{pancyclic} if it contains a cycle of every length from 3 to \(n\). The graph \(\Gamma\) is \emph{vertex pancyclic} (resp., \emph{edge pancyclic}) if every vertex (resp., edge) lies on a cycle of each length from 3 to \(n\). Clearly, if \(\Gamma\) is edge pancyclic, it is also vertex pancyclic; if \(\Gamma\) is vertex pancyclic, it is pancyclic; and if \(\Gamma\) is pancyclic, it is Hamiltonian.
		
		Let \(G\) be a finite group, and let \(S\) be an inverse-closed subset of \(G\) with \(1 \notin S\). The \emph{Cayley graph} \(\Gamma(G, S)\) is the graph with elements of \(G\) as vertices, where two vertices \(u, v \in G\) are connected by an edge if and only if \(v = su\) for some \(s \in S\). \(\Gamma(G, S)\) is connected if and only if \(S\) is a set of generators for \(G\), and it is vertex-transitive.
		
		Let \(S_n\) denote the symmetric group on \([n] = \{1, \ldots, n\}\). Let \(D_n\) be the set of all derangements in \(S_n\), where a derangement is a permutation with no fixed points. The number of derangements is given by 
		\[
		|D_n| = n! \sum_{i=0}^{n} \frac{(-1)^i}{i!}.
		\]
		The \emph{derangement graph} \(\Gamma_n\) is the Cayley graph \(\Gamma(S_n, D_n)\), where two vertices \(g, h \in \Gamma_n\) are adjacent if and only if \(g(i) \neq h(i)\) for all \(i \in [n]\), or equivalently, if \(h^{-1}g\) fixes no points. Note that \(\Gamma_n\) is loopless because \(D_n\) does not contain the identity of \(S_n\), and it is simple because \(D_n\) is inverse-closed, i.e., \(D_n = \{g^{-1} \mid g \in D_n\}\). Clearly, \(\Gamma_n\) is vertex-transitive, and therefore \(|D_n|\)-regular. Moreover, \(\Gamma_n\) is connected for \(n \geq 4\) because every vertex can be reached from the identity.
		
		In recent decades, many studies have investigated the edge-pancyclicity and edge-fault-tolerant pancyclicity of Cayley graphs on symmetric groups. For example, Jwo et al.\cite{Jwo} and Tseng et al.\cite{Tseng} examined the bipancyclicity and edge-fault-tolerant bipancyclicity of star graphs. Kikuchi and Araki\cite{Kikuchi} discussed the edge-bipancyclicity and edge-fault-tolerant bipancyclicity of bubble-sort graphs. Tanaka et al.\cite{Tanaka} studied the bipancyclicity of Cayley graphs generated by transpositions. 
		
		The derangement graph has also been extensively studied. Research on \(\Gamma_n\) includes topics such as its independence number \cite{KLW, M1}, EKR property \cite{MS}, eigenvalues \cite{KW18, KW13, KW10, R07}, and automorphism group \cite{DZ11}, among other properties. A significant area of interest is the Hamiltonian property of \(\Gamma_n\). The question of whether the derangement graph is Hamiltonian was posed in \cite{R84, W88}, and the existence of a Hamiltonian cycle was proven in \cite{M85, W89}. In \cite{RS}, Rasmussen and Savage showed that \(\Gamma_n\) is Hamilton-connected, meaning that every pair of distinct vertices is connected by a Hamiltonian path. 
		
		In this paper, we establish the following results:

		\begin{thm}\label{main1}
				The derangements graph $\Gamma_n$ is edge pancyclic for $n\ge 4$.
		\end{thm}
		
		Then we have the following corollary directly.
		
		\begin{cor}\label{cor1}
		The derangements graph $\Gamma_n$ is $($vertex$)$ pancyclic for $n\ge 4$.
		\end{cor}
		
		We can generalize the above results in two directions. First, let $D_n$ be the set of all permutations with no fixed points, we can see if a constant number of fixed points is permitted, the resulting Cayley graph is still edge-pancyclic.
		
		Fix a non-negative integer $k$, let $D^k_n$ be the set of all permutations with exactly $k$ fixed points. And $\Gamma^k_n$ is a short for the Cayley graph $\Gamma(S_n,D^k_n)$. When $k=0$, $D^0_n=D_n$ is the derangement of $S_n$, and $\Gamma^0_n=\Gamma_n$ is the derangement graph. We have the following results.
		
		\begin{thm}\label{main2}
			For any integer $k\geq 0$, when $n\geq 2k+1$ and $n\geq 4$, $\Gamma^k_n$ is edge-pancyclic.
		\end{thm}
		
		Let $A_{[n]}^k$ denote the ordered $k$-tuples with points in $[n]$.
		For any $k\geq 4$ and $n\geq k$, we denote $G_n^k$ as the graph with vertex set $A_{[n]}^k$, and two vertices $a=(a_1,\dots,a_k)$, $b=(b_1,\dots,b_k)$ are adjacent if $a_i\neq b_i$ for $i=1,\dots,k$. Notice that when $n=k$, $G^k_n\cong \Gamma_n$.
		We will see when $n>k$, the resulting graph $G^k_n$ is still edge-pancyclic.
		
		\begin{thm}\label{main3}
				When $n\geq k\geq 4$, $G_n^k$ is edge-pancyclic.
			\end{thm}
		
		The paper is arranged as follows. In Section \ref{sec2}, we will prove Theorem \ref{main1}. In Sections \ref{sec3} and \ref{sec4}, we will give the generalization of Theorems \ref{main1} and prove Theorem \ref{main2} and \ref{main3}, respectively.

		\vskip.2cm
		
		\section{Proof of Theorem \ref{main1}}\label{sec2}
		\vskip.2cm
		
		In order to proof Theorem \ref{main1}, we need the following lemma.
		
		\vskip.2cm
		\begin{lemma}\label{lem1}{\em (\cite{M85},\text{Theorem 45})}
		Let $\Gamma$ be a graph of order $n \ge 3$. If
			$\delta(\Gamma)\ge (n+2)/2$, then $\Gamma$ is edge pancyclic.
			\end{lemma}
		
		We need some extra notations. Let $S_n$ be the symmetric group on $[n] = \{1,\ldots, n\}$. We denote by $C_n$ the set of permutations in $S_n$ that consist of one single cycle of
		length $n$. We call these {\em cyclic permutations}. It is clear that $|C_n| = (n -1)!$ and $\{1,\sigma(1),\sigma^2(1),\ldots,\sigma^{n-1}(1)\}=[n]$ for $\sigma\in C_n$. For $\sigma_1,\sigma_2\in S_n$, let $\Delta(\sigma_1,\sigma_2)$ be the numbers of the fixed points of $\sigma_1^{-1}\sigma_2$. We first have the following claim.
		
		\begin{claim}\label{claim2.1}
			 For any $\alpha,\beta\in S_n$ and $\sigma\in C_n$, we have $$\sum_{i=0}^{n-1}\Delta(\alpha,\sigma^i\beta)=n.$$
			\end{claim}
		
		\noindent{\bf Proof of Claim~\ref{claim2.1}} Note that for any $a,b\in [n]$, there is only $i\in \{0,1,\ldots,n-1\}$ such that $\sigma^i\beta(a)=b$. Since $\sigma\in C_n$, $\{\sigma^i\beta(a)~|~i=0,1,\ldots,n-1\}=[n]$ which implies the result holds.\q
		
		\vskip.2cm
		
		\noindent{\bf Proof of Theorem \ref{main1}} Given $\alpha\beta\in E(\Gamma_n)$. By the definition of $\Gamma_n$, we have $\alpha\not=\beta$ and $ \Delta(\alpha,\beta)=0$. Since $|C_n|=(n-1)!$ and $n\ge 4$, there is $\sigma\in C_n$ such that $\alpha\beta^{-1}\not\in\{\sigma,\sigma^2,\ldots,\sigma^{n-1}\}$. By Claim \ref{claim2.1} and $ \Delta(\alpha,\beta)=0$, there is $i_0\in [n-1]$ such that $ \Delta(\alpha,\sigma^{i_0}\beta)\ge 2$. Since $\alpha\beta^{-1}\not\in\{\sigma,\sigma^2,\ldots,\sigma^{n-1}\}$, $\alpha\not=\sigma^{i_0}\beta$. Let $\beta_0=\sigma^{i_0}\beta$ for short. Then there are $a,b,c,d\in [n]$ such that $\alpha(a)=\beta_0(a)=b$ and $\alpha(c)=\beta_0(c)=d$. Assume, without loss of generality, that $c=n$.
		
		\begin{claim}\label{claim2.2}
			We can assume that $d=n$.
			\end{claim}
		
		\noindent{\bf Proof of Claim \ref{claim2.2}} Assume $d\not=n$.  Let $\gamma$ be a transposition $(d,n)$ in $S_n$.
		Denote a mapping $\varphi~:~V(\Gamma_n)\rightarrow V(\Gamma_n)$ such that $\varphi(\theta)=\gamma\theta$. Since $\alpha^{-1}\beta=\alpha^{-1}\gamma^{-1}\gamma\beta=(\gamma\alpha)^{-1}(\gamma\beta)$, $\varphi$ is an automorphism of $\Gamma_n$ which implies the claim holds. \q
		
		By Claim \ref{claim2.2}, we assume $\alpha(n)=\beta_0(n)=n$. Denote $T=\{\tau\in S_n~|~\tau(n)=n\}$. Then $|T|=(n-1)!$ and $\alpha,\beta_0\in T$. For any $\tau\in T$, let $A_\tau=\{\tau,\sigma\tau,\sigma^2\tau,\ldots,\sigma^{n-1}\tau\}$. Then $\beta\in A_{\beta_0}$.
		
		\begin{claim}\label{claim2.3}
		$S_n=\cup_{\tau\in T}A_\tau$ and $\Gamma_n[A_\tau]\cong K_n$ for any $\tau\in T$, where $K_n$ is a complete graph of order $n$.
		\end{claim}
		
		\noindent{\bf Proof of Claim \ref{claim2.3}} In order to show $S_n=\cup_{\tau\in T}A_\tau$, we just need to prove $A_{\tau_1}\cap A_{\tau_2}=\emptyset$ for any $\tau_1,\tau_2\in T$ with $\tau_1\not=\tau_2$. Suppose there are $\tau_1,\tau_2\in T$ with $\tau_1\not=\tau_2$ such that $A_{\tau_1}\cap A_{\tau_2}\not=\emptyset$. Then there are $i,j\in \{0,1,\ldots,n-1\}$ such that $\sigma^i\tau_1=\sigma^j\tau_2$. Assume $i>j$. Then we have $\sigma^{i-j}\tau_1=\tau_2$. Since $\tau_1,\tau_2\in T$, we have $\tau_1(n)=\tau_2(n)=n$ which implies $\sigma^{i-j}(n)=n$, a contradiction with $\sigma\in C_n$.
		
		Let $\tau\in T$ and $\pi_1,\pi_2\in A_\tau$. Then there are $i,j\in \{0,1,\ldots,n-1\}$ such that $\pi_1=\sigma^i\tau$ and $\pi_2=\sigma^j\tau$. Assume $i>j$. If $\Delta(\pi_1,\pi_2)\ge 1$, say $\pi_1(k)=\pi_2(k)$ ($k\in [n]$), then $\sigma^{i-j}(\tau(k))=\tau(k)$, a contradiction with $\sigma\in C_n$. Hence $\Delta(\pi_1,\pi_2)=0$ and then $\Gamma_n[A_\tau]\cong K_n$.\q

		\begin{claim}\label{claim2.4}
		 Let $\overline{\Gamma_{n-1}}$ be the complement of $\Gamma_{n-1}$. If $n\ge 5$, then $\overline{\Gamma_{n-1}}$ is edge pancyclic. If $n=4$, then $\overline{\Gamma_{3}}$ is edge even-pancyclic. 
		 \end{claim}
		
		\noindent{\bf Proof of Claim \ref{claim2.4}} Note that $\Gamma_{n-1}$ is $|D_{n-1}|$-regular. Then
		$$\delta(\overline{\Gamma_{n-1}})=(n-1)!-1-(n-1)!\sum_{i=0}^{n-1}\frac{(-1)^i}{i!}=(n-1)!\sum_{i=1}^{n-1}\frac{(-1)^{i-1}}{i!}-1\ge \frac{(n-1)!}{2}+1$$ if $n\ge 5$. Thus $\overline{\Gamma_{n-1}}$ is edge pancyclic by Lemma \ref{lem1} when $n\ge 5$.
		
		If $n=4$, then $\overline{\Gamma_{3}}\cong K_{3,3}$. Thus the result holds.
		\q
		
		We complete the proof by considering the following two cases.
		
		\noindent{\bf Case 1.} $n\ge 5$.
		
		Let $\tau=\tau(1)\tau(2)\cdots\tau(n-1)\tau(n)\in T$. Then $\tau(n)=n$. Denote $\widehat{\tau}=\tau(1)\cdots\tau(n-1)$ and $\widehat{T}=\{\widehat{\tau}~|~\tau\in T\}$. Then $\widehat{\tau}\in S_{n-1}$ and $\widehat{T}=S_{n-1}$. So $\Gamma(\widehat{T},D_{n-1})=\Gamma_{n-1}$.
		Since $\alpha,\beta_0\in T$ and $ \Delta(\alpha,\beta_0)\ge 2$, $ \Delta(\widehat{\alpha},\widehat{\beta_0})\ge 1$ which implies $\widehat{\alpha}\widehat{\beta_0}\in E(\overline{\Gamma(\widehat{T},D_{n-1})})$. By Claim \ref{claim2.4}, for any integer $3\le k\le (n-1)!$, there are $\widehat{\tau_1},\widehat{\tau_2},\ldots,\widehat{\tau_k}\in \widehat{T}$ such that $\widehat{\tau_1}=\widehat{\alpha}$, $\widehat{\tau_2}=\widehat{\beta_0}$ and $\widehat{\tau_1}\widehat{\tau_2}\ldots\widehat{\tau_k}\widehat{\tau_1}$ is a cycle of $\overline{\Gamma(\widehat{T},D_{n-1})}$.
		Since $\widehat{\tau_i}\widehat{\tau_{i+1}}\in E(\overline{\Gamma(\widehat{T},D_{n-1})})$,
		$\Delta(\widehat{\tau_i},\widehat{\tau_{i+1}})\ge 1$ for all $1\le i\le k$, where the subscripts are modulo $k$. Hence $\Delta(\tau_i,\tau_{i+1})\ge 2$ for all $1\le i\le k$. By Claim \ref{claim2.1}, there is $\theta_{i+1}\in A_{\tau_{i+1}}$ such that $\Delta(\tau_i,\theta_{i+1})=0$ which implies $\tau_i\theta_{i+1}\in E(\Gamma_n)$ for all $1\le i\le k$. Recall $\tau_1=\alpha$, $\tau_2=\beta_0$, $\beta\in A_{\beta_0}$ and $ \Delta(\alpha,\beta)=0$. Then we can let $\theta_2=\beta$. By Claim \ref{claim2.3}, $\Gamma_n[A_{\tau_i}]\cong K_n$ for all $1\le i\le k$ and they are vertex-disjoint. Let $P_{ij}$ be a path of length $j$ connecting $\theta_i$ and $\tau_i$ in $\Gamma_n[A_{\tau_i}]$, where $1\le i\le k$ and $1\le j\le n-1$. Then
		$$\theta_1P_{1j_1}\tau_1(=\alpha)\theta_2(=\beta)P_{2j_2}\tau_2\theta_3P_{3j_3}\ldots \theta_kP_{kj_k}\tau_k\theta_1$$is a cycle of length $k+\sum_{s=1}^kj_s$ contained the edge $\alpha\beta$, where $1\le j_s\le n-1$ for all $1\le s\le k$. Since $3\le k\le (n-1)!$, there is a cycle of length $l$ contained  $\alpha\beta$ for all $6\le l\le n!$. To finish our proof, we just need to show that there is a cycle of length $l$ contained  $\alpha\beta$ for all $3\le l\le 5$.
		
		For any $\pi\in S_n$, denote $M(\pi)=\{(1,\pi(1)),(2,\pi(2)),\ldots,(n,\pi(n))\}$. Then there is a bijection between $\pi$ and $M(\pi)$. For any $\pi_1,\pi_2\in V(\Gamma_n)$, if $\pi_1\pi_2\in E(\Gamma_n)$, then $M(\pi_1)\cap M(\pi_2)=\emptyset$; vice versa. Particularly, $M(\alpha)\cap M(\beta)=\emptyset$.  We consider the complete bipartite graph $K_{n,n}$. Then $M(\pi)$ can be treated as a perfect matching of $K_{n,n}$. Since $n\ge 5$, we can find five disjoint perfect matchings $M(\alpha), M(\beta),M(\pi_1),M(\pi_2),M(\pi_3)$. Hence $\alpha\beta\pi_1\alpha$, $\alpha\beta\pi_1\pi_2\alpha$ and $\alpha\beta\pi_1\pi_2\pi_3\alpha$ are three cycles contained  $\alpha\beta$ of length 3, 4 and 5, respectively.
		
		\begin{figure}[!htbp]
			\begin{center}\includegraphics[scale=0.8]{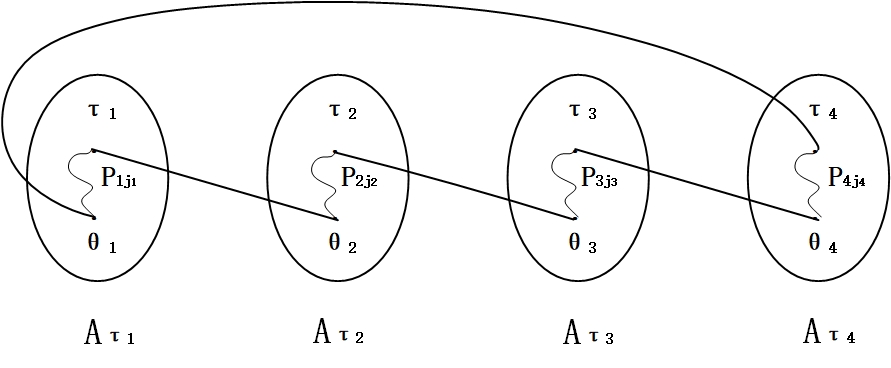}\\\end{center}
			\caption{The construction of the cycle when $k=4$}\label{k4}
		\end{figure}
		
		\noindent{\bf Case 2.} $n=4$.
		
		Let $\tau=\tau(1)\tau(2)\tau(3)\tau(4)\in T$. Then $\tau(4)=4$. Denote $\widehat{\tau}=\tau(1)\tau(2)\tau(3)$ and $\widehat{T}=\{\widehat{\tau}~|~\tau\in T\}$. By the same argument as that of Case 1, we have $\Gamma(\widehat{T},D_{3})=\Gamma_{3}$. By Claim \ref{claim2.4}, for  $ k=4$ and $6$, there are $\widehat{\tau_1},\widehat{\tau_2},\ldots,\widehat{\tau_k}\in \widehat{T}$ such that $\widehat{\tau_1}=\widehat{\alpha}$, $\widehat{\tau_2}=\widehat{\beta_0}$ and $\widehat{\tau_1}\widehat{\tau_2}\ldots\widehat{\tau_k}\widehat{\tau_1}$ is a cycle of length $k$ in $\overline{\Gamma(\widehat{T},D_{3})}$. Then there is a cycle of length $l$ contained  $\alpha\beta$ for all $8\le l\le 24$. By the same argument, we can find four disjoint perfect matchings $M(\alpha), M(\beta),M(\pi_1),M(\pi_2)$. Hence $\alpha\beta\pi_1\alpha$ and $\alpha\beta\pi_1\pi_2\alpha$ are two cycles contained  $\alpha\beta$ of length 3 and 4, respectively. Now we consider $A_\alpha$ and $A_{\beta_0}$. Then $\beta\in A_{\beta_0}$. Recall $\beta_0=\sigma^{i_0}\beta$. Since $\Delta(\alpha,\beta)=0$, we have $\Delta(\sigma^{i_0}\alpha,\sigma^{i_0}\beta)=0$ which implies $\alpha_0\beta_0\in \Gamma_4$, where $\alpha_0=\sigma^{i_0}\alpha$. Since $\alpha_0\in A_\alpha$ and $|A_\alpha|=|A_{\beta_0}|=4$, we easily have cycles of length 5 to 7 contained  $\alpha\beta$ by Claim \ref{claim2.3}.
		
		Thus we complete the proof.\qed
		
		\section{Proof of Theorem \ref{main2}}\label{sec3}
		In this section, we will prove $\Gamma^k_n=\Gamma(S_n, D^k_n)$ is edge-pancyclic, where $D^k_n$ is the set of all permutations with exactly $k$ fixed points.
		Let $A_{[n]}^k$ be the $k$-tuples with points in $[n]$. For any $\theta\in A_{[n]}^k$, the notation $\{\theta\}$ is to view $\theta$ as a set. And for any $\sigma_1,\sigma_2\in A_{[n]}^k$, let $\Delta(\sigma_1,\sigma_2)$ be the number of $i\in[k]$ such that $\sigma_1(i)=\sigma_2(i)$.
		
		For any edge $e=\alpha \beta$ in $\Gamma^k_n$, according to the definition of $\Gamma^k_n$, there are exactly $k$ fixed points in $\alpha^{-1}\beta$.
		
		\begin{claim}\label{claim3.1}
		 We can assume the fixed points of $\alpha^{-1}\beta$ are in the position $n-k+1,n-k+2,\dots,n$.
		 \end{claim}
		
		\noindent{\bf Proof of Claim \ref{claim3.1}.} 
		Suppose the index of the fixed points of $\alpha^{-1}\beta$ is $I=\{i_1,\dots,i_k\}$. Then there exist a permutation $\gamma\in S_n$ such that $\gamma(i_j)=n-k+j$ for $j=1,\dots,k$.
		
		Denote a mapping $\phi~:~ V(\Gamma^k_n)\to V(\Gamma^k_n)$ such that $\phi(\theta)=\gamma\theta$. Since $\alpha^{-1}\beta=\alpha^{-1}\gamma^{-1}\gamma\beta=(\gamma\alpha)^{-1}(\gamma\beta)$, $\varphi$ is an automorphism of $\Gamma_n$ which implies Claim \ref{claim3.1} holds.\q
		
		For every $\eta\in A_{[n]}^k$, we set $A_\eta$ be the collection of all the permutations ended with $\eta$, which implies every $\theta\in A_\eta$, we have $\theta(n-k+j)=\eta(j)$ for every $j=1,\dots,k$. Notice that $\Gamma^k_n[A_\eta]\cong \Gamma_{n-k}$ for every $\eta\in A_{[n]}^k$, every edge in $\Gamma^k_n$ is contained in cycles of each length in $[3,(n-k)!]$.
		
		\begin{claim}\label{claim3.2}
	 For $k\geq 1$ and $n\geq 2k+1$, there exist an order $\left\{\eta_1,\dots,\eta_{k!{n\choose k}}\right\}$ of $A_{[n]}^k$
			such that $\Delta(\eta_i,\eta_{i+1})=0$ and $|\{\eta_i\}\cap\{\eta_{i+1}\}|\geq k-1$ for $i=1,\dots,n!/(n-k)!$.
			\end{claim}
		
		\noindent{\bf Proof of Claim \ref{claim3.2}.}
		First, we can order ${n\choose k}$ sets of ${[n]\choose k}=\{\gamma_1,\gamma_2,\dots,\gamma_{{n\choose k}}\}$ such that $|\gamma_i\cap\gamma_{i+1}|=k-1$ for $i=1,\dots,{n\choose k}-1$. This can be proved by induction.
		Actually, we can prove a stronger result, which also requires that $\gamma_1=\{1,\dots,k\}$ and $\gamma_{n\choose k}=\{n-k+1,\dots,n\}$. When $k=1,2$, it is easy to check. When $k\geq 3$, by induction hypothesis, there exist an order $\tau_{1},\tau_2,\dots, \tau_{n-1\choose k-1}$ of the set ${[n]\setminus\{1\}\choose k-1}$, with $\tau_1=\{2,3,\dots,k\}$, $\tau_{n-1\choose k-1}=\{k-n+2,\dots,n\}$ and $|\tau_i\cap\tau_{i+1}|=k-1$ for $i=1,\dots,{n-1\choose k-1}$. Let $\gamma_i=\{1\}\cup\tau_i$ for $i=1,\dots,{n-1\choose k-1}$ and $\gamma_{{n-1\choose k-1}+1}=\tau_{n-1\choose k-1}\cup\{2\}$. 
		
		By induction hypothesis and the symmetry, there exists an order $\tau_1',\dots,\tau_{n-2\choose k-1}'$ or $\{[n]\setminus\{1,2\}\}$ such that $\tau_1'=\{n-k+1,\dots,n\}$, $\tau_{n-2\choose k-1}'=\{3,4,\dots,k+1\}$ with $|\tau_i'\cap\tau_{i+1}'|=k-1$ for $i=1,\dots,{n-2\choose k-1}$. Let $\gamma_{{n-1\choose k-1}+i}=\tau'_i\cup\{2\}$ for $i=1,\dots,{n-2\choose k-1}$. Repeat this process until we have $\gamma_{n\choose k}=\{n-k+1,\dots,n\}$. Then we find the order of ${[n]\choose k}$ we want.
		
		Since $\Gamma_k$ is vertex-pancyclic, there exists an order of $\eta_1,\dots,\eta_{k!}$ such $\{\eta_j\}=\gamma_1$ for $j=1,\dots,k!$, and $\Delta(\eta_i,\eta_{i+1})=0$ for $i=1,\dots,k!$.
		
		Now suppose $a_1=\gamma_1\setminus\gamma_2$ and $b_1=\gamma_2\setminus\gamma_1$. Let $\tilde{\eta_{k!}}$ be the $k$-tuple that replace $a_1$ with $b_1$, which implies $\{\tilde{\eta_{k!}}\}=\gamma_2$. Then we have $\Delta(\eta_{k!},\sigma(\tilde{\eta_{k!}}))=0$ where $\sigma\in C_k$ is a cyclic permutation. Repeat this process ${n\choose k}$ times and we will find the $\eta_1,\dots,\eta_{n!/(n-k)!}$ we want.
		\q
		
		Notice that in the proof of Claim \ref{claim3.2}, we suppose $\{\eta_1\}=\{1,2,\dots,k\}$, but according to the symmetry, we can suppose $\eta_1$ be any $k$-tuple in $A_{[n]}^k$.
		
		\noindent{\bf Proof of Theorem \ref{main2}.}
		According to Claim \ref{claim3.1}, we may assume the edge $\alpha\beta\in E(\Gamma^k_n[A_{\eta_1}])$, and $\eta_1,\eta_2,\dots,\eta_{n!/(n-k)!}$ are as in Claim \ref{claim3.2}. Then for $3\leq \ell \leq (n-k)!$, there is a cycle $C^\ell$ of length $\ell$ containing $\alpha\beta$. Let $(\epsilon_1,\eta_1)(\tau_1,\eta_1)\in E(C^\ell)$ different from $\alpha\beta$, where $\Delta(\epsilon_1,\tau_1)=0$.
		
		If $\{\eta_1\}=\{\eta_2\}$, let $\pi\in S_{n-k}$ with exactly $k$ fixed points, then $\Delta(\pi(\epsilon_1),\pi(\tau_1))=0$ and $(\epsilon_1,\eta_1)(\pi(\epsilon_1),\eta_2)(\pi(\tau_1),\eta_2)(\tau_1,\eta_1)(\epsilon_1,\eta_1)$ is a cycle of length $4$. Moreover, $(\pi(\epsilon_1),\eta_2)(\pi(\tau_1)\in E(\Gamma^k_n[A_{\eta_2}])$. Thus, for $3\leq\ell \leq (n-k)!$, there exists a cycle of length $\ell$ containing $(\pi(\epsilon_1),\eta_2)(\pi(\tau_1),\eta_2)$. By adding the edges $(\epsilon_1,\eta_1)(\pi(\epsilon_1),\eta_2),(\pi(\tau_1),\eta_2)(\tau_1,\eta_1) $ and deleting the edges $(\epsilon_1,\eta_1)(\tau_1,\eta_1),(\pi(\epsilon_1),\eta_2)(\pi(\tau_1),\eta_2)$, we can integrate two cycles in $\Gamma_n^k[A_{\eta_1}]$ and $\Gamma_n^k[A_{\eta_2}]$ respectively into one longer cycle. Thus $\alpha\beta$ lies in cycles of each length in $[3,2(n-k)!]$.
		
		If $|\{\eta_1\}\cap\{\eta_2\}|=k-1$, suppose $a_1=\{\eta_1\}\setminus\{\eta_2\}$ and $b_1=\{\eta_2\}\setminus\{\eta_1\}$. Let $\tilde{\epsilon_1}$ be the $k$-tuple replace the point $b_1$ with $a_1$, and $\tilde{\tau_1}$ be the $k$-tuple replace the point $a_1$ with $b_1$. Let $\pi\in S_{n-k}$ be a permutation such that $\Delta(\epsilon_1,\pi(\tilde{\epsilon_1}))=k$. Then $\Delta(\tau_1,\pi(\tilde{\tau_1}))=k$ and $\Delta(\pi(\tilde{\epsilon_1}),\pi(\tilde{\tau_1}))=0$. Since $n\geq 2k+1$, such $\pi$ is exist. Thus $(\epsilon_1,\eta_1)(\pi(\tilde{\epsilon_1}),\eta_2)(\pi(\tilde{\tau_1}),\eta_2)(\tau_1,\eta_1)(\epsilon_1,\eta_1)$ is a cycle of length $4$. We can similarly prove that $\alpha\beta$ is contained in cycles of each length in $[3,2(n-k)!]$.
		
		For any edge in $E(\Gamma_n^k[A_{\eta_2}])$ different from $(\pi(\epsilon_1),\eta_2)(\pi(\tau_1),\eta_2)$, we can repeat the above process and find a cycle of length $4$ between $\Gamma_n^k[A_{\eta_2}]$ and $\Gamma_n^k[A_{\eta_3}]$. Then we will prove $\alpha\beta$ contained in cycles of each length in $[3,3(n-k)!]$. Repeat this process we can prove $\alpha\beta$ is contained in cycles of each length in $[3,n!]$, which implies the result holds.
		\qed
		\begin{figure}[h]
			\centering
			\includegraphics[width=0.8\linewidth]{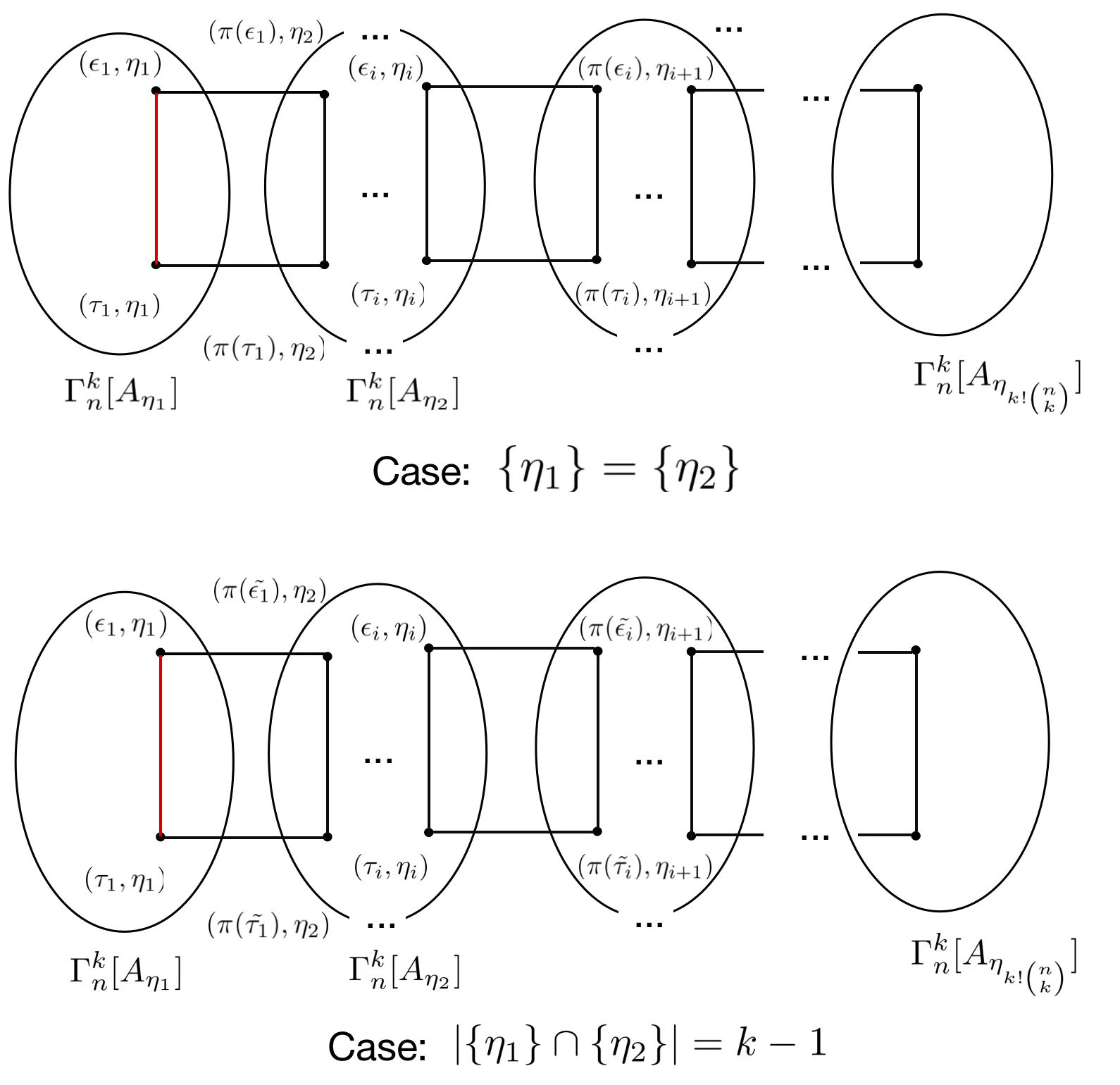}
			\caption{The Construction of cycles in $\Gamma_n^k$}
			\label{fig:enter-label}
		\end{figure}
		\section{Proof of Theorem \ref{main3}}\label{sec4}
		We will prove Theorem \ref{main3} by induction on $n$. When $n=k$, $G^k_n\cong\Gamma_k$, which is edge-pancyclic, is established.
		Now consider the case when $n>k$. Fix any point $i\in[n]$, let $E_i$ denote the collection of $k$-tuples in $A_{[n]}^k$ that contains $i$, and $F_i$ denote the collection of $k$-tuples in ${[n]\choose k}$ that do not contain $i$. Then set $H^i_1:=G^k_n[E_i]$ and $H^i_2:=G^k_n[F_i]$, we have that $H^i_2\cong G_{n-1}^k$. By induction hypothesis, we have that $G_i^2$ is edge-pancyclic.
		
		Then we will prove that $H^i_1$ is also edge-pancyclic with a similar method as the proof of Theorem \ref{main1}.
		
		In this section, $C_k$ denotes the set of cyclic permutations of $S_k$. And for any $\sigma_1,\sigma_2\in A_{[n]}^k$, let $\Delta(\sigma_1,\sigma_2)$ be the number of $i\in[k]$ such that $\sigma_1(i)=\sigma_2(i)$. We have the following claim.
		
		\begin{claim}\label{claim4.1}
		For any $\alpha,\beta\in A_{[n]}^k$ and $\sigma\in C_n$, we have
			$$\sum_{i=1}^k\Delta(\alpha,\sigma^i\beta)=|\{\alpha\}\cap\{\beta\}|.$$ 
			\end{claim}
		
		\noindent{\bf Proof of Claim \ref{claim4.1}.}
		Note that for any $a\in[k]$ and $b\in\{\alpha\}\cap\{\beta\}$, there is only one $i\in\{0,1,\dots,k-1\}$ such that $\sigma^i\beta(a)=b$, which implies the result holds.
		\q
		
		\begin{lemma}\label{lem2}
			When $n>k\geq 4$, the graph $H^i_1$ is edge-pancyclic for $i=1,2,\dots,n$.
		\end{lemma}
		\begin{proof}
		For convenience, we suppose $i=n$. For any edge $\alpha\beta$ in $H^n_1$, we have $n\in \{\alpha\}\cap \{\beta\}$ by the definition of $H^n_1$. If $\{\alpha\}=\{\beta\}$, then as we explained in Section 2, there exist $\beta_0=\sigma^{i_0}\beta$ for some $i_0\in [k]$ such that $\Delta(\alpha,\beta_0)\geq 2$, and we may assume $\alpha(k)=\beta_0(k)=n$.
		If $\{\alpha\}\neq \{\beta\}$, we may assume $\alpha(k)=n$ and then there exist $\beta_0=\sigma^{i_0}\beta$ for some $i_0\in[k]$ such that $\beta_0(k)=n$ and $\Delta(\alpha,\beta_0)\geq 1$.
		
		Denote $T=\{\tau\in A_{[n]}^k~|~ \tau(k)=n\}$, then we have $\alpha,\beta_0\in T$. For any $\tau\in T$, let $A_\tau=\{\tau,\sigma\tau,\sigma^2\tau,\dots,\sigma^{k-1}\tau\}$. Then $\beta\in A_{\beta_0}$. Similarly to Claim \ref{claim2.3} in Section \ref{sec2}, we have the following claim.
		
		\begin{claim}\label{claim4.3}
		 $A_n=\cup_{\tau\in T}A_\tau$ and $G^k_n[A_\tau]\cong K_k$ for any $\tau\in T$.
		\end{claim}
		
		Now we construct a new graph $\tilde{G_1}$ with vertex set as $A_{[n-1]}^{k-1}$ and two vertices $\sigma_1,\sigma_2$ are adjacent if $\{\sigma_1\}\neq \{\sigma_2\}$ or if $\Delta(\sigma_1,\sigma_2)\geq 2.$
		
		\begin{claim}\label{claim4.4}
		If $n> k\geq 4$, then $\tilde{G_1}$ is edge-pancyclic.
		\end{claim}
		
		\noindent{\bf Proof of Claim \ref{claim4.4}.} 
		
		Note that $\overline{\tilde{G}}$ is $|D_{k-1}|$-regular, then
		\begin{align*}
			\delta(\tilde{G})=(n-1)!/(n-k)!-(k-1)!\sum_{i=0}^{k-1}\frac{(-1)^i}{i!}\geq& (n-1)!/(n-k)!\sum_{i=1}^{k-1}\frac{(-1)^{i-1}}{i!}-1\\
			\geq& \frac{(n-1)!/(n-k)!}{2}+1
		\end{align*}		
		if $n> k\geq 4$. Thus $\tilde{G_1}$ is edge-pancyclic by Lemma \ref{lem1}.
		\q
		
		Then we will finish the proof of Lemma \ref{lem2}. Let $\tau\in T$, then $\tau(k)=n$. Denote $\hat{\tau}=\tau(1)\dots \tau(k-1)$, and $\hat{T}=\{\hat{\tau}~|~\tau\in T\}$. Since $\alpha,\beta_0\in T$ and $\Delta(\alpha,\beta_0)\geq 2, \Delta(\hat{\alpha},\hat{\beta_0})\geq 1$ or $\{\hat{\alpha}\}\neq \{\hat{\beta_0}\}$,  $\hat{\alpha}\hat{\beta_0}\in E(\tilde{G_1})$. By Claim \ref{claim4.4}, for any integer $\ell\in [3,(n-1)!/(n-k)!]$, there exist $\hat{\tau_1}(=\hat{\alpha}),\hat{\tau_2}(=\hat{\beta_0})\dots,\tau_\ell\in \hat{T}$ that construct a cycle in $\tilde{G_1}$ as $\hat{\tau_1}\dots\hat{\tau_{\ell}}\hat{\tau_1}$. If $\Delta(\hat{\tau_i},\hat{\tau_{i+1}})\geq 1$, then $\Delta(\tau_i,\tau_{i+1})\geq 2$ and by Claim \ref{claim4.1}, there exist $\theta_{i+1}\in A_{\tau_{i+1}}$ such that $\Delta(\tau_i,\theta_{i+1})=0$ which implies $\tau_i\theta_{i+1}\in E(H^n_1)$. If $\{\hat{\tau_i}\}\neq \{\hat{\tau_{i+1}}\},$
		then $|\{\tau_i\}\cap\{\tau_{i+1}\}|\leq k-1$, by Claim \ref{claim4.1}, there exist $\theta_{i+1}\in A_{\tau_{i+1}}$ such that $\Delta(\tau_i,\theta_{i+1})=0$ which implies $\tau_i\theta_{i+1}\in E(H^n_1)$. 
		
		Recall that $\tau_1=\alpha,\tau_2=\beta$, $\beta\in A_{\beta}$ and $\Delta(\alpha,\beta)=0$, we can similarly find the cycle of length $\ell+\sum_{s=1}^\ell j_s$ containing the edge $\alpha\beta$, where $1\leq j_s\leq k-1$ for all $1\leq s\leq \ell$. And then $\alpha\beta$ is contained in cycles of each length in $[6,k(n-1)!/(n-k)!]$. Since $|V(H^n_1)|=k(n-1)!/(n-k)!$, it left to proof $\alpha\beta$ in the cycles of each length in $[3,5]$. This can be similarly proved as the proof of Theorem \ref{main1}.
		\end{proof}
		
		\noindent{\bf Proof of Theorem \ref{main3}.}
		Now for any fixed edge $\alpha\beta$, we will prove $\alpha\beta$ is contained in cycles of each length in $[3,n!/(n-k)!]$. We finish the proof by considering the following two cases.
		
		\noindent{\bf Case 1.}
		$\{\alpha\}=\{\beta\}$.
		
		Then we may assume $\{\alpha\}=[k]$. Thus the edge $\alpha\beta\in E(H^k_1)$. According to Lemma \ref{lem2}, $\alpha\beta$ is contained in cycles of each length in $[3,k(n-1)!/(n-k)!]$. For a cycle with length $\ell$ containing $\alpha\beta$ which is denoted as $C^\ell$, choose an edge $e=\sigma_1\sigma_2\in E(C^\ell)$  that is distinct from $\alpha\beta$. Let $\tilde{\sigma_i}$ be the permutation in $A_{[n]}^k$ that replace the point $k\in\sigma_i$ with $k+1$ for $i=1,2$.
		
		We have $\sigma(\tilde{\sigma_1})\sigma(\tilde{\sigma_2})\in E(H^k_2)$ and $\sigma_i\sigma(\tilde{\sigma_i})\in E(G^k_n)$ for $i=1,2$, where $\sigma\in C_k$ is a cyclic permutation in $S_k$. Since $H_2^k\cong G_{n-1}^k$ and by induction hypothesis, the edge $\sigma(\tilde{\sigma_1})\sigma(\tilde{\sigma_2})$ is contained in cycles of each length in $[3,(n-1)!/(n-1-k)!]$. By deleting the edges $\sigma_1\sigma_2,\sigma(\tilde{\sigma_1})\sigma(\tilde{\sigma_2})$ and adding the edges $\sigma_i\sigma(\tilde{\sigma_i})$, $i=1,2$, we can integrate two cycles contained in $H_1^k$ and $H_2^k$ respectively into one longer cycles. Thus the edge $\alpha\beta$ is contained in cycles of each length in $[3,n!/(n-k)!]$.
		
		\noindent{\bf Case 2.}
		$\{\alpha\}\neq\{\beta\}$.
		
		We may assume $k\in\{\alpha\}\setminus\{\beta\}$. Thus the edge $\alpha\sigma(\alpha)\in E(H_1^k)$ and $\beta\sigma(\beta)\in E(H^k_2)$, moreover, $\sigma(\alpha)\sigma(\beta)\in E(G^k_n)$. By Lemma \ref{lem2} and the induction hypothesis that $H_2^k$ is edge-pancyclic, the edge $\alpha\sigma(\alpha)$ is contained in cycles of each length in $[3,k(n-1)!/(n-k)!]$, and the edge $\beta\sigma(\beta)$ is contained in cycles of each length in $[3,(n-1)!/(n-1-k)!]$. 
		
		Then by deleting the edges $\alpha\sigma(\alpha)$, $\beta\sigma(\beta)$ and adding the edges $\alpha\beta$, $\sigma(\alpha)\sigma(\beta)$, we can  integrate two cycles in $H_1^k$ and $H_2^k$ respectively into one longer cycle containing $\alpha\beta$. Thus the edge $\alpha\beta$ is contained in cycles of each length in $[4,n!/(n-k)!]$. With a similar proof as in Theorem \ref{main1}, the edge $\alpha\beta$ is contained in a $C_3$. We have finished the proof. \qed

		\section*{Acknowledgement}
		M. Cao is supported by the National Natural Science Foundation of China (12301431). M. Lu is supported by the National Natural Science Foundation of China (12171272).
		
		\section*{Declaration of competing interest}
The authors declare that they have no known competing financial interests or personal relationships that could have appeared to influence the work reported in this paper.

	\section*{Data availability}
No data was used for the research described in the article.

		\vskip.2cm


\vskip.2cm
		\begin{thebibliography}{99}
			
			\bibitem{DZ11} Y. Deng, X. Zhang,  Automorphism group of the derangement graphs, Electron. J. Combin. 18 (2011),  \#R198.
			
			\bibitem{Jwo} J. Jwo, S. Lakshmivarahan, and S. Dhall, Embedding of cycles and grids in star graphs, J. Circuits, Syst. Comput. 1 (1991), 43-74.
			
			\bibitem{KW18} Ch. Ku, K. Wong, Eigenvalues of the matching derangement graph, J. Algebraic Combin.  48 (2018),  627-646.
			
			\bibitem{KLW} Ch. Ku, T. Lau, K. Wong,  Largest independent sets of certain regular subgraphs of the derangement graph, J. Algebraic Combin.  44 (2016),  81-98.
			
			\bibitem{KW13} Ch. Ku, K. Wong, Solving the Ku-Wales conjecture on the eigenvalues of the derangement graph, European J. Combin. 34 (2013), 941-956.
			
			
			\bibitem{KW10} Ch. Ku, D. Wales, Eigenvalues of the derangement graph, J. Combin. Theory Ser. A  117 (2010),  289-312.
		
			
			\bibitem{Kikuchi} Y. Kikuchi and T. Araki, Edge-bipancyclicity and edge-fault-tolerant bipancyclicity of bubble-sort graphs, Inform. Process. Lett. 100 (2006), 52-59.
			
			\bibitem{M1} K. Meagher, A. Razafimahatratra, P. Spiga,  On triangles in derangement graphs, J. Combin. Theory Ser. A  180  (2021), Paper No. 105390.
			
			
			\bibitem{MS} K. Meagher, P. Sin,
			All 2-transitive groups have the EKR-module property,
			J. Combin. Theory Ser. A 177 (2021), Paper No. 105322.
			
			\bibitem{M85} J. Metzger, Problem 1186, Math. Mag. 58 (1985) 113-114.
			
			\bibitem{R84} S. Rabinowitz, Problem 1186, Math. Mag. 57 (1984) 109.
			
			
			
			\bibitem{R07} P. Renteln, On the spectrum of the derangement graph, Electron. J. Combin. 14 (2007),  \#R82.
			
			\bibitem{RS} D.J. Rasmussen, C.D. Savage, Hamilton-connected derangement graphs on $S_n$, Discrete Math.  133 (1994),  217-223.
			
			\bibitem{Tanaka} Y. Tanaka, Y. Kikuchi, T. Araki, and Y. Shibata, Bipancyclic properties of Cayley graphs generated by transpositions, Discret. Math. 310 (2010), 748-754.
			
			\bibitem{Tseng} Y.C. Tseng, S.H. Chang, and J.P. Sheu, Fault-tolerant ring embedding in a star graph with both link and node failures, IEEE Trans. Parallel Distrib. Syst. 8 (1997), 1185-1195.
			
			\bibitem{W88} H.S. Wilf, Generalized Gray codes, invited talk, SIAM Conf. on Discrete Math. San Francisco,
			June 1988.
			
			
			\bibitem{W89} H.S. Wilf, Combinatorial Algorithms: An Update (SIAM, Philadelphia 1989).
			
			
			
		\end{thebibliography}
	\end{document}